\documentclass[a4paper]{article}
\usepackage{fullpage}
\usepackage{amsmath}
\usepackage{amssymb}
\usepackage{amsthm}
\usepackage{bbm}
\usepackage{enumitem}
\usepackage[colorlinks=true,urlcolor=blue,linkcolor=blue,plainpages=false,pdfpagelabels]{hyperref}

\emergencystretch=\maxdimen
\newtheorem{theorem}{Theorem}[section]

\newtheorem{lemma}[theorem]{Lemma}
\newtheorem{corollary}[theorem]{Corollary}

\newtheorem{notation}[theorem]{Notation}
\newtheorem{definition}[theorem]{Definition}

\title{Proof mining in $L^p$ spaces}
\author{Andrei Sipo\c s${}^{a,b}$\\[0.2cm]
\footnotesize ${}^a$Department of Mathematics, Technische Universit\"at Darmstadt,\\
\footnotesize Schlossgartenstrasse 7, 64289 Darmstadt, Germany\\[0.1cm]
\footnotesize ${}^b$Simion Stoilow Institute of Mathematics of the Romanian Academy,\\
\footnotesize Calea Grivi\c tei 21, 010702 Bucharest, Romania \\[2mm]
\footnotesize E-mail: sipos@mathematik.tu-darmstadt.de
}
\date{}

\begin{document}

\maketitle

\begin{abstract}
We obtain an equivalent implicit characterization of $L^p$ Banach spaces that is amenable to a logical treatment. Using that, we obtain an axiomatization for such spaces into a higher-order logical system, the kind of which is used in proof mining, a research program that aims to obtain the hidden computational content of mathematical proofs using tools from mathematical logic. As an aside, we obtain a concrete way of formalizing $L^p$ spaces in positive-bounded logic. The axiomatization is followed by a corresponding metatheorem in the style of proof mining. We illustrate its use with the derivation for this class of spaces of the standard modulus of uniform convexity.

\noindent 2010 {\it Mathematics Subject Classification}: 03F10; 46B25; 46E30.

\noindent {\em Keywords:} Proof mining; $L^p$ spaces; Uniform convexity; Metatheorems.
\end{abstract}

\section{Introduction}

Since the time of Hilbert and Bernays \cite{HilBer68}, a notable research topic has been the search for the proper way of interfacing logic with analysis. One of the first methods to represent real numbers in a logic in a built-in manner was attempted in the 1960s -- see, e.g., the book on continuous model theory by Chang and Keisler \cite{ChaKei66}. Later, Ben Yaacov and others realized that the lack of fruitful lines of research out of that logic was due to an unfortunate choice of parameters -- specifically, the truth values could vary widely along an arbitrary compact Hausdorff space (instead of just the interval $[0,1]$), while equality itself was tightly restricted to binary values. Their efforts led to what has been called ``continuous first-order logic'', a system in which many celebrated and relatively advanced results of 20th century model theory could be reasonably translated -- see \cite{BenBerHenUsv08} for an introduction. Another strand of developments came from Henson's positive-bounded logic, introduced in \cite{Hen76} and later shown to be largely equivalent to continuous first-order logic. Despite this fact, due to its later exhaustive treatment by Henson and Iovino focusing on the model-theoretic ultraproduct construction \cite{HenIov02}, positive-bounded logic was subject to a thorough investigation from which it resulted that, in combination with the aforementioned ultraproducts, it could be used to prove uniformity results in nonlinear analysis and ergodic theory -- see the recent paper of Avigad and Iovino \cite{AviIov13}.

What interests us here is the other known method of obtaining such results, namely the research program of ``proof mining'' -- a project first suggested by G. Kreisel in the 1950s under the name of ``unwinding of proofs'' and then given maturity by U. Kohlenbach and his collaborators starting in the 1990s. Proof mining aims to analyse existing proofs in branches of ordinary mathematics in order to exhibit their hidden combinatorial and computational content and also to devise general ``metatheorems'' \cite{GerKoh06,Koh05,GerKoh08} that explain when such concrete witnesses or bounds may be extracted from a known proof, conditional on the formalization of the given proof inside some higher-order system of arithmetic. These general results may also specify when some parameters do not partake in the final formula -- hence the uniformity result is implicit in the quantitative one. So far, the research has been largely focused on nonlinear analysis and, naturally, the question of the right way of formalizing metric or normed spaces has been raised. Fortunately, the higher-order nature of the systems with which proof mining works has provided the following fourth solution to the problem: spaces are encoded as separate primitive types out of which the type algebra is constructed and on which axioms like the Banach space ones can be added as simply as in the purely arithmetic situation. A comprehensive introduction to the theory of proof mining and its results up to 2008 is \cite{Koh08}, while recent surveys may be found in \cite{Koh17, Koh19}.

We can now ask the question of whether these proof-theoretic methods are sufficiently powerful to provide us with all uniformity results given to us by the model-theoretic properties of positive-bounded formulas. (Proof theory already had the upper hand in the matter of being able to deal with weak forms of extensionality.) The answer, as presented in the 2016 paper of G\"unzel and Kohlenbach \cite{GunKoh16}, is in the affirmative. To give a rough sketch, the positive-bounded formulas are there translated into a special class of higher-order formulas denoted by $\mathcal{PBL}$, which are then turned into $\Delta$-formulas, a class of formulas which can be freely added as additional axioms, with no negative consequences to the bound extraction procedure, as per the classical metatheorems of proof mining. A new metatheorem is then obtained for the classes of spaces which could be axiomatized by positive-bounded formulas. In addition, the treatment of a ``uniform boundedness principle'' tries to clarify just what exactly is the role played by the ultraproduct construction. Examples are given of such classes of spaces, and the translations for each set of axioms into the higher-order language are given explicitly, together with their metatheorems. Notable among these are the $L^p$ and $BL^pL^q$ Banach lattices, which are usually defined by a construction, but for which axiomatic characterizations into positive-bounded logic have been found, for the last one by Henson and Raynaud \cite{HenRay07}.

The space of $p$-integrable functions on a measure space $(\Omega,\mathcal{F},\mu)$ -- denoted by $L^p(\Omega,\mathcal{F},\mu)$ or simply by $L^p(\mu)$ -- is the Banach space built on the set of all real-valued measurable functions $f$ on $\Omega$ having the property that
$$\int_\Omega |f|^p \mathrm{d}\mu < \infty,$$
a set then factored by the a.e.-equality relation (which makes the canonical seminorm into a norm). It turns out -- see \cite{Lac74,LinTza73} for detailed expositions -- that these spaces can be given an implicit characterization, which resembles a bit the axiomatization of $BL^pL^q$ lattices which was analysed by G\"unzel and Kohlenbach. Notably, and in contrast to that case, our characterization does not use at all the natural lattice structure. What we shall do is to show how it may be modified in order to build from it a logical system that (i) accurately represents the $L^p(\mu)$ spaces relatively to their standard models (Theorem~\ref{sound}); (ii) allows for a bound extraction metatheorem (Theorem~\ref{S3thm}); and (iii) admits an internal proof that the standard modulus of uniform convexity is valid for this class of spaces (Theorem~\ref{uconv}).

We shall draw, in Theorem~\ref{pb}, a similar conclusion to that of \cite{HenRay07}: the class of $L^p$ spaces is axiomatizable in the language of positive-bounded logic within the class of pure Banach spaces, therefore obtaining a new, concrete proof of this classical result which was previously obtained by Krivine \cite{Kri74} using an alternative method of characterization and by Henson \cite{Hen76} using ultraproducts.

The next section runs parallel to the exposition in \cite{Koh08} (as updated by \cite{GunKoh16}) and familiarizes the reader with the basic notions regarding the logical system used and the formulation of metatheorems, adapted for the present situation. Section~\ref{sec:delta} introduces and proves the essential lemmas that we use to obtain our characterization, presents and justifies its translation into the higher-order language and gives the corresponding metatheorem. Finally, Section~\ref{sec:uc} exhibits an illustration of the use of our axiomatization, namely the way in which one can derive the uniform convexity of $L^p$ spaces, for the case $p \geq 2$.

\section{Logical preliminaries}

The most powerful foundational system that has so far been studied from the viewpoint of proof mining is the system $\mathcal{A}^\omega$ of weakly extensional classical analysis in all finite types. It is created by adjoining certain choice principles -- the quantifier-free axiom schema of choice and the axiom schema of dependent choice -- to a base system of higher-order arithmetic, namely weakly extensional classical (Peano) arithmetic in all finite types, a system that is a modification of G\"odel's System {\bf T}, which is equivalent in proof-theoretic power to the first-order theory {\bf PA}. The system $\mathcal{A}^\omega$, in its turn, is of a power comparable to the first-order, two-sorted theory usually denoted by $\text{\bf Z}_\text{\bf 2}$ or ``full second-order arithmetic''. A detailed presentation of it can be found in the monograph \cite{Koh08}. We note in passing that it represents real numbers by functions $\mathbb{N} \to \mathbb{N}$ in such a way that the usual binary relations $=_\mathbb{R}$ and $\leq_\mathbb{R}$ are actually expanded into purely universal formulas, and $<_\mathbb{R}$ into a purely existential one. However, situations like the one happening in the usual definition of a convergent sequence, where ``$<\varepsilon$'' can be readily substituted for ``$\leq\varepsilon$'', occur frequently, this giving us a leeway in minimizing the complexity of the formulas under discussion.

From the way it has been built up, it is immediate that this system admits a G\"odelian functional interpretation in its bar-recursive extension devised by Spector \cite{Spe62}. However, in order to be useful to actual applications, we must do slight modifications of it, as it has been done for the general logical metatheorems of proof mining, developed by Kohlenbach \cite{Koh05} and by Gerhardy and Kohlenbach \cite{GerKoh06,GerKoh08}. We follow in the sequel the exposition of \cite{Koh08}, in order to present the first such extension, $\mathcal{A}^\omega[X,\|\cdot\|]$, which allows us to speak about normed spaces.

The set of types for this system, $\text{{\bf T}}^X$, will be generated by two ``primitive'' types, the type $0$ of natural numbers and a new abstract type $X$, representing elements from our space, forming a free algebra with a single binary operation $\to$, representing function types. (We will write $\tau(\rho)$ for $\rho\to\tau$.) For such a type $\rho$, we define the type $\hat{\rho}$ by replacing all occurences of $X$ in $\rho$ by $0$.

\begin{definition}
Such a type is {\bf small} if it is of the form $\rho\underbrace{(0)\ldots(0)}_{n\text{ times}}$, where $\rho \in \{0,X\}$ and $n \geq 0$.
\end{definition}

\begin{definition}
Such a type is {\bf admissible} if it is of the form $\rho(\tau_n)\ldots(\tau_1)$, where $\rho \in \{0,X\}$, $n \geq 0$ and $\tau_1$, ..., $\tau_n$ are all small.
\end{definition}

Clearly, all small types are admissible.

Also, we add new constants for the various operations common to normed spaces, i.e. $0_X$ and $1_X$ of type $X$, $+_X$ of type $X(X)(X)$, $-_X$ of type $X(X)$, $\cdot_X$ of type $X(X)(1)$ (where $1=0(0)$ is the type of real numbers) and $\|\cdot\|_X$ of type $1(X)$. We allow infix notation and the ``syntactic sugar'' of writing $x -_X y$ for $x +_X (-_X y)$. Finally, we add the following axioms:
\begin{enumerate}
\item the equational, and hence purely universal, axioms for vector spaces;
\item $\forall x^X (\|x-_X x\|_X =_\mathbb{R} 0_\mathbb{R} )$;
\item $\forall x^X y^X (\|x -_X y\|_X =_\mathbb{R} \|y -_X x\|_X)$;
\item $\forall x^Xy^Xz^X (\|x -_X z\|_X \leq_\mathbb{R} \|x-_X y\|_X +_\mathbb{R} \|y-_Xz\|_X)$;
\item $\forall \alpha^1 x^X y^X (\|\alpha x-_X\alpha y\|_X =_\mathbb{R} \|\alpha\|_\mathbb{R} \cdot_\mathbb{R} \|x-_X y\|_X$;
\item $\forall \alpha^1 \beta^1 x^X (\|\alpha x -_X \beta x\|_X =_\mathbb{R} |\alpha -_\mathbb{R} \beta |_\mathbb{R} \cdot_\mathbb{R} \|x\|_X$;
\item $\forall x^X \forall y^X \forall u^X \forall v^X (\|(x+_X y) -_X (u+_X v)\|_X \leq_\mathbb{R} \|x -_X u\|_X +_\mathbb{R} \|y-_X v\|_X)$;
\item $\forall x^X y^X (\|(-_X x) -_X (-_X y)\|_X =_\mathbb{R} \|x-_X y\|_X)$;
\item $\forall x^X y^X (|\|x\|_X -_\mathbb{R} \|y\|_X|_\mathbb{R} \leq_\mathbb{R} \|x-_X y\|_X)$;
\item $\|1_X\|_X =_\mathbb{R} 1_\mathbb{R}$.
\end{enumerate}

Note that the equality relation $x^X =_X y^X$ which is necessarily used in the expression of the vector space axioms is syntactically defined as $\|x -_X y\|_X =_\mathbb{R} 0_\mathbb{R}$. We define the equality for higher types as in the system $\mathcal{A}^\omega$, as extensional equality reducible to $=_0$ and $=_X$.

An issue when adding new constant symbols is their extensionality -- roughly, as the base system admits only a quantifier-free rule of extensionality, it is not clear that for a new function symbol $f$ that is added to the system (e.g. $+_X$ or $-_X$ from the above) one can prove in the new system a statement of the form
$$\forall x_1...\forall x_n\forall y_1...\forall y_n \left(\bigwedge_i x_i = y_i \to f(x_1,...,x_n) = f(y_1,...,y_n)\right)$$
Some axioms above, like the eighth one, are written in this way purely to minimize the effort in writing such an extensionality proof; the rest of them yield it more readily in their classical forms. The result is that all new function symbols are provably extensional. The last axiom was originally added solely to ensure the non-triviality of the formalized space, but, as we shall see later, we can make good use of it in bounding some quantities that will appear in the axioms we propose.

In order to formalize the fact that the space is Banach, i.e. its completeness, the following is done (see \cite[pp. 432-434]{Koh08}). We first note that the following operation on $X$-valued sequences is term-definable in the system:
\[
 \widehat{x_n} :=
  \begin{cases} 
      \hfill x_n, \hfill &  \text{if, for all $k <n$, } [\widehat{d_X(x_k, x_{k+1})}](k+1) <_\mathbb{Q} 6 \cdot 2^{-k-1}\\
      \hfill x_k, \hfill & \text{where $k < n$ is the least such that } [\widehat{d_X(x_k, x_{k+1})}](k+1) \geq_\mathbb{Q} 6 \cdot 2^{-k-1}\\
  \end{cases}
\]
where we have used explicitly the encoding of reals as functions. The operation above transforms a sequence into a Cauchy one of prescribed rate $2^{-n+3}$. We now add a new constant $C$ of type $X(X(0))$, used to assign the limit to such sequences. This is enforced by the following additional axiom:
$$\forall x^{X(0)} \forall k^0 (d_X(C(x), \widehat{x_k}) \leq_\mathbb{R} 2^{-k+3}).$$
We have therefore obtained the system $\mathcal{A}^\omega[X,\|\cdot\|,\mathcal{C}]$, formalizing Banach spaces.

We say that a formula in our language is a $\forall$-formula (resp. an $\exists$-formula) if it is formed by adjoining a list of universal (resp. existential) quantifiers over variables of admissible types to a quantifier-free formula.

Now, if $(X,\|\|)$ is a Banach space, we define a canonical associated set-theoretic model $\mathcal{S}^{\omega,X} = \{\mathcal{S}_\rho\}_{\rho \in \text{{\bf T}}^X}$ in all finite types by putting $\mathcal{S}_0 := \mathbb{N}$, $\mathcal{S}_X := X$ and $\mathcal{S}_{\tau(\rho)} := \mathcal{S}_{\tau}^{\mathcal{S}_\rho}$ (i.e. the set-theoretic Hom-set), assigning to any language constant its standard value, except for $1_X$, which can take any value of norm $1$ -- this is why we said ``a'' set-theoretic model. In order to define e.g. the norm function, one assigns to each real a type-$1$ functional (obviously non-effectively), as described by \cite[Definition 17.7]{Koh08}. Also, we say that a sentence of our logical language is modeled by such a pair $(X,\|\|)$ iff it is satisfied in the usual Tarskian sense by all the possible models associated to it (i.e., regardless of the exact value of $1_X$, which, however, makes the tenth axiom to be satisfied in this sense).

There is another relevant model associated to this kind of logical system. In order to introduce it, we define, for each $\rho \in \text{{\bf T}}^X$, the majorization relation $\gtrsim_\rho \subseteq \mathcal{S}_{\widehat{\rho}} \times \mathcal{S}_\rho$, inductively, as follows:
\begin{align*}
x^* \gtrsim_0 x &:\Leftrightarrow x^* \geq x\\
x^* \gtrsim_X x &:\Leftrightarrow x^* \geq \|x\|\\
x^* \gtrsim_{\tau(\rho)} x &:\Leftrightarrow \forall y^*,y (y^* \gtrsim_\rho y \to x^*y^* \gtrsim_\tau xy)\land \forall y^*,y (y^* \gtrsim_{\widehat{\rho}} y \to x^*y^* \gtrsim_{\widehat{\tau}} x^*y).
\end{align*}
We can now define the model of {\bf hereditarily strongly majorizable functionals}, $\mathcal{M}^{\omega,X} = \{\mathcal{M}_\rho\}_{\rho \in \text{{\bf T}}^X}$, by:
\begin{align*}
\mathcal{M}_0 &:= \mathbb{N} \\
\mathcal{M}_X &:= X \\
\mathcal{M}_{\tau(\rho)} &:= \{x \in \mathcal{M}_\tau^{\mathcal{M}_\rho} \mid \text{exists } x^* \in \mathcal{M}_{\widehat{\tau}}^{\mathcal{M}_{\widehat{\rho}}} \text{ such that } x^* \gtrsim_{\tau(\rho)} x \}
\end{align*}

One of the main uses of this majorizable model arises from the fact that, unlike the standard model, it is a model of bar recursion, which is needed in the current state of the art to interpret the principle of dependent choice. Therefore, the proof of the general logical metatheorems involves some constant juggling between the two models (see \cite[pp. 421-428]{Koh08}). As a consequence, the kind of sentences that one may freely add as axioms will be restricted here not only by the logical complexity, but also by the types involved. Here we see how the admissible types come into play -- for such a type $\rho$, it is the fact (see \cite[Lemma 5.7]{GunKoh16}) that $\mathcal{M}_\rho \subseteq \mathcal{S}_\rho$. This justifies the following definition.

\begin{definition}
We say that a formula in our system is a {\bf $\Delta$-sentence} if it is of the following form:
$$\forall\underline{a}^{\underline{\delta}} \exists \underline{b}^{\underline{\sigma}} \preceq_{\underline{\sigma}} \underline{r} \underline{a} \forall \underline{c}^{\underline{\gamma}} B_0(\underline{a},\underline{b},\underline{c}),$$
where underlined letters represent tuples of variables or types, $B_0$ is quantifier-free and devoid of any additional variables, $\underline{r}$ is a term tuple of the appropriate type, $\underline{\delta}$, $\underline{\sigma}$, $\underline{\gamma}$ are tuples of admissible types, and $\preceq$ is syntactic sugar for the following family of binary relations:
\begin{align*}
x \preceq_0 y &:\equiv x \leq y \\
x \preceq_X y &:\equiv \|x\| \leq \|y\| \\
x \preceq_{\tau(\rho)} y &:\equiv \forall z^\rho (x(z) \preceq_\tau y(z))
\end{align*}
\end{definition}

\begin{definition}
The {\bf Skolem normal form} of a $\Delta$-sentence written as above is:
$$\exists \underline{B}^{\underline{\sigma}(\underline{\delta})} \preceq_{\underline{\sigma}(\underline{\delta})} \underline{r} \forall \underline{a}^{\underline{\delta}} \forall \underline{c}^{\underline{\gamma}} B_0(\underline{a}, \underline{B}\underline{a}, \underline{c}) $$
\end{definition}

\begin{notation}
If $\Delta$ is a set of $\Delta$-sentences, we denote by $\widetilde{\Delta}$ the set of the Skolem normal forms of the sentences in the set $\Delta$.
\end{notation}

\begin{theorem}[{\cite[Lemma 5.11]{GunKoh16}}]
Let $(X,\|\|)$ be a Banach space, $\mathcal{S}^{\omega,X}$ and $\mathcal{M}^{\omega,X}$ be models associated with it as above. Let $\Delta$ be a set of $\Delta$-sentences. Suppose that $\mathcal{S}^{\omega,X} \models \Delta$. Then $\mathcal{M}^{\omega,X} \models \widetilde{\Delta}$.
\end{theorem}

The following result is the appropriate modification for our case of \cite[Theorem 5.13 and Corollary 5.14]{GunKoh16}.

\begin{theorem}[{Logical metatheorem for Banach spaces endowed with additional $\Delta$-axioms}]\label{S2thm}\ \\
Let $\rho \in \text{{\bf T}}^X$ be an admissible type. Let $B_\forall(x,u)$ be a $\forall$-formula with at most $x,u$ free and $C_\exists(x,v)$ an $\exists$-formula with at most $x,v$ free. Let $\Delta$ be a set of $\Delta$-sentences. Suppose that:
$$\mathcal{A}^\omega[X,\|\cdot\|,\mathcal{C}] + \Delta \vdash \forall x^\rho (\forall u^0 B_\forall(x,u) \to \exists v^0 C_\exists(x,v)).$$
Then one can extract a partial functional $\Phi : S_{\widehat{\rho}} \rightharpoonup \mathbb{N}$, whose restriction to the strongly majorizable functionals of $S_{\widehat{\rho}}$ is a bar-recursively computable functional of $\mathcal{M}^\omega$, such that for all Banach spaces $(X,\|\|)$ having the property that any associated set-theoretic model of it satisfies $\Delta$, we have that for all $x \in S_\rho$ and $x^* \in S_{\widehat{\rho}}$ such that $x^* \gtrsim_\rho x$, the following holds:
$$\forall u \leq \Phi(x^*) B_\forall(x,u) \to \exists v \leq \Phi(x^*) C_\exists(x,v).$$
In addition:
\begin{enumerate}[label=\arabic*.]
\item If $\widehat{\rho}$ is equal to $1$, then $\Phi$ is total.
\item All variables may occur as finite tuples satisfying the same restrictions.
\item If the proof in the system above proceeds without the use of the axiom of dependent choice, one can use solely the set-theoretical model $\mathcal{S}^{\omega, X}$, without any restriction to the majorizable functionals, and $\Phi$ is then a total computable functional which is higher-order (i.e. in the sense of G\"odel) primitive recursive. Also, the additional restriction imposed on $\rho$ is no longer necessary.
\end{enumerate}
\end{theorem}

\section{\texorpdfstring{The $\Delta$-axiomatization of $L^p(\mu)$ Banach spaces}{The axiomatization of Lp Banach spaces}}\label{sec:delta}

The goal of this section is to describe an extension of the theory in the previous section, one that can formalize the concept of an $L^p(\mu)$ Banach space. Since such spaces are usually defined explicitly, as equivalence classes of $p$-integrable real-valued functions on a measurable space, it is clear that an implicit characterization is needed. Such a characterization in terms of the natural lattice structure of $L^p(\mu)$ spaces was used in \cite{GunKoh16} in order to provide a logical metatheorem for this class of spaces. For our use, however, it is more helpful to use the following characterization, for which references are \cite{LinTza73, Lac74} and which uses solely the Banach space structure. In the sequel, we shall denote by $\mathbb{R}^n_p$ the Euclidean vector space $\mathbb{R}^n$ endowed with the standard $p$-norm.

\begin{definition}
Let $X$ and $Y$ be linearly isomorphic Banach spaces. The {\bf Banach-Mazur distance} between $X$ and $Y$ is
$$d(X,Y) := \inf \{ \|T\| \|T^{-1}\| \mid T \text{ is a linear isomorphism between $X$ and $Y$}\}.$$
\end{definition}

\begin{definition}
Let $p$, $\lambda > 1$. We say that a Banach space $X$ is an {\bf $\mathcal{L}_{p,\lambda}$ space} if for each finite dimensional subspace $Y$ of $X$ there exists a finite dimensional subspace $Z$ of $X$ such that $Y \subseteq Z$ and $d(Z,\mathbb{R}^{\dim_\mathbb{R} Z}_p) \leq \lambda$.
\end{definition}

\begin{theorem}[{\cite{LinPel68, Tza69}}]\label{Banach-char}
Let $p > 1$. A Banach space $X$ is isometric to some $L_p(\mu)$ space iff for all $\varepsilon > 0$, $X$ is an $\mathcal{L}_{p,1+\varepsilon}$ space.
\end{theorem}

In order to obtain an axiomatization that is amenable to a logical treatment in the spirit of the previous section, we shall modify the above characterization in order to add more quantitative information. The first, crucial step is the following lemma.

\begin{lemma}\label{l1}
Let $X$ be the $L^p$ space on a measure space $(\Omega, \mathcal{F}, \mu)$. Then, for all $x_1$,..., $x_n$ in $X$ of norm at most $1$, and for all $N \in \mathbb{N}_{\geq 1}$, there is a subspace $C \subseteq X$ and $y_1$,..., $y_n$ in $C$ such that $C$ is of dimension at most $(2nN + 1)^n$, it is isometric to $\mathbb{R}^{\dim_\mathbb{R} C}_p$ and for all $i$, $\|x_i - y_i\| \leq \frac1N$.
\end{lemma}

The above lemma, although with a bound having a more convoluted expression, may be obtained as an immediate consequence of \cite[Theorem 2.1]{PelRos74}. Later in that same paper (see \cite[p. 269]{PelRos74}) a proof which yields a bound of $(2n(n+1)N)^n$ is briefly sketched. In order to obtain our bound, we present the following simplified proof which uses an argument adapted from the proof of \cite[Proposition 3.7]{HenRay07}.

\begin{proof}[{Proof of Lemma~\ref{l1}}]
For any $f : \Omega \to \mathbb{R}$, we denote by $|f| : \Omega \to \mathbb{R}$ the function defined, for all $\omega \in \Omega$, by $|f|(\omega) := |f(\omega)|$.

We fix from the beginning some representatives for $x_1$,..., $x_n$, denoting them by the same designators, and we note that all constructions below will be well-defined w.r.t. the a.e.-equality equivalence relation. We set $\varphi:=\sum_{j=1}^n |x_j|$ and, for each $i \in \{1,\ldots,n\}$ and $k \in \{0,\ldots,nN-1\}$:
$$A_{i,k}:=\left\{\omega \in \Omega \mid \frac{k}{nN}\varphi(\omega) < |x_i(\omega)| \leq \frac{k+1}{nN}\varphi(\omega)\right\},$$
$$A_{i,k,+} := \{\omega \in A_{i,k} \mid x_i(\omega) > 0\}, \quad A_{i,k,-} := \{\omega \in A_{i,k} \mid x_i(\omega) < 0\},$$
$$A_{i,\otimes} := \{\omega \in \Omega \mid x_i(\omega) = 0\}.$$

Clearly, for all $i$, we have that $\Omega = \bigcup_{k=0}^{nN-1} (A_{i,k,+} \cup A_{i,k,-}) \cup A_{i, \otimes}$ and this is a disjoint union in all of its components. 

For each $i$, put $y_i := \sum_{k=0}^{nN-1} \frac{k}{nN} (\mathbbm{1}_{A_{i,k,+}}-\mathbbm{1}_{A_{i,k,-}}) \cdot \varphi$. Let $i \in \{1,\ldots,n\}$ and $\omega \in \Omega$ be such that $x_i(\omega) > 0$. Then, by the above, there is a unique $k_0$ such that $\omega \in A_{i,k_0,+}$ and there is not any $k$ such that $\omega \in A_{i,k,-}$. Therefore, $y_i(\omega) = \frac{k_0}{nN} \cdot \varphi(\omega)$. As $\omega \in A_{i,k_0,+}$, $x_i (\omega) \leq \frac{k_0+1}{nN}\varphi(\omega)$, so $x_i (\omega) - y_i(\omega) \leq \frac{\varphi(\omega)}{nN}$. Since we also have that $x_i (\omega) > \frac{k_0}{nN}\varphi(\omega) = y_i(\omega)$ (so $x_i(\omega) - y_i(\omega)>0$), we get that $|x_i(\omega) - y_i(\omega)|\leq \frac{\varphi(\omega)}{nN}$. Analogously, we might prove this result for $x_i(\omega) = 0$ and $x_i(\omega) < 0$. We have therefore established that for all $i$, $|x_i - y_i| \leq \frac1{nN} \cdot \varphi$. From that we get that for all $i$,
$$\|x_i - y_i \| \leq \frac1{nN} \cdot \|\varphi\| \leq \frac1{nN} \sum_{j=1}^n\|x_j\| \leq \frac1N.$$

Returning to the disjoint union from before, we remark that, for different $i$'s, those sets might overlap. Therefore, for each $l : \{1,\ldots,n\} \to ((\{0,\ldots,nN-1\} \times \{+,-\}) \cup \{\otimes\})$, set:
$$B_l := \bigcap_{i=1}^n A_{i,l(i)}$$
so
$$\Omega = \bigcup_l B_l$$
is a disjoint union. For each such $l$, of which there are $(2nN + 1)^n$, set now:
$$z_l := \mathbbm{1}_{B_l} \cdot \varphi.$$
We have, then, for each $i$, that:
\begin{align*}
y_i &= \sum_{k=0}^{nN-1} \frac{k}{nN} (\mathbbm{1}_{A_{i,k,+}}-\mathbbm{1}_{A_{i,k,-}}) \cdot \varphi \\
&= \sum_{k=0}^{nN-1} \frac{k}{nN} \left(\sum_{l(i)=(k,+)}\mathbbm{1}_{B_l}-\sum_{l(i)=(k,-)}\mathbbm{1}_{B_l}\right) \cdot \varphi \\
&= \sum_{k=0}^{nN-1} \frac{k}{nN} \left(\sum_{l(i)=(k,+)}z_l - \sum_{l(i)=(k,-)}z_l\right),
\end{align*}
i.e. a linear combination of $z_l$'s.

Let $D$ be the set of all $l$'s such that $z_l \neq 0$. We take $C$ to be the space spanned by all the $z_l$'s with $l \in D$. It clearly contains, by the above, all the $y_i$'s and is of dimension at most (actually, equal, as we shall see) the cardinality of $D$, which is in turn at most $(2nN + 1)^n$. It remains to show that it is isometric to $\mathbb{R}^D_p$. If $l \in D$, then:
$$0 \neq \|z_l\| = \left(\int_\Omega |z_l|^p \mathrm{d}\mu\right)^{\frac1p} =  \left(\int_{B_l} |\varphi|^p \mathrm{d}\mu\right)^{\frac1p}.$$
We can now show that the linear map $f:\mathbb{R}^D_p \to C$, defined on the standard basis vectors by $f(e_l):=\frac1{\|z_l\|}\cdot z_l$ is an isometry. Let $v \in \mathbb{R}^D_p$, so there exist $(\lambda_l)_{l \in D}$ such that $v = \sum_{l \in D} \lambda_l e_l$. Then we have that:
\begin{align*}
\|f(v)\| &= \left\|\sum_{l \in D} \frac{\lambda_l}{\|z_l\|}\cdot z_l \right\| \\
&= \left(\int_\Omega \left| \sum_{l \in D} \frac{\lambda_l}{\|z_l\|} \cdot \mathbbm{1}_{B_l} \cdot \varphi \right|^p \mathrm{d}\mu\right)^{\frac1p} \\
&= \left(\sum_{l \in D} \int_{B_l} \left|\frac{\lambda_l}{\|z_l\|}\right|^p \cdot \left|\varphi \right|^p \mathrm{d}\mu\right)^{\frac1p} &\text{(as the $B_l$'s are disjoint)} \\
&= \left(\sum_{l \in D} \left|\frac{\lambda_l}{\|z_l\|}\right|^p \int_{B_l} \left|\varphi \right|^p \mathrm{d}\mu\right)^{\frac1p}  \\
&= \left(\sum_{l \in D} \left|\lambda_l\right|^p \right)^{\frac1p}  \\
&= \left\| \sum_{l \in D} \lambda_l e_l \right\| \\
&= \|v\|,
\end{align*}
and we are done.
\end{proof}

\begin{lemma}\label{l1-2}
The statement of Lemma~\ref{l1} is still valid if we require that all $y_i$'s are of norm at most $1$ and we allow for $C$ to be of dimension at most $(4nN + 1)^n$.
\end{lemma}

\begin{proof}
We apply Lemma~\ref{l1} for our $x_i$'s, but with $N$ replaced by $2N$. We therefore obtain a subspace $C \subseteq X$ and $y'_1$,..., $y'_n$ in $C$ such that $C$ is of dimension at most $(4nN + 1)^n$, it is isometric to $\mathbb{R}^{\dim_\mathbb{R} C}_p$ and for all $i$, $\|x_i - y'_i\| \leq \frac1{2N}$. For each $i$, if $\|y'_i\|\geq 1$, set $y_i := \frac{y'_i}{\|y'_i\|}$, else put $y_i := y'_i$. For the ``unmodified'' $y_i$'s, clearly $\|x_i - y_i\| \leq \frac1N$. The others are certainly still in $C$, so we must only show for them that $\|x_i - y_i\| \leq \frac1N$.

Set $\alpha_i := \frac1{\|y'_i\|}$. Since $\|y'_i\| \leq \|x_i\| + \|y'_i - x_i\| \leq 1 + \frac1{2N}$, we get that $\frac{1-\alpha_i}{\alpha_i} \leq \frac1{2N}$, so:
$$\|x_i - y_i\| = \|x_i - \alpha_iy'_i\| \leq \|x_i - y'_i\| + \|y'_i - \alpha_iy'_i\| \leq \frac1{2N} + (1-\alpha_i)\|y'_i\| =  \frac1{2N} + \frac{1-\alpha_i }{\alpha_i} \leq \frac1N,$$
and we are done.
\end{proof}

\begin{lemma}\label{l2-new}
Let $X$ be a Banach space that satisfies the conclusion of Lemma~\ref{l1-2}. Then, for all $x_1$,..., $x_n$ in $X$ of norm exactly $1$, and for all $N \in \mathbb{N}_{\geq 1}$, there is a subspace $C \subseteq X$ and $y_1$,..., $y_n$ of norm exactly $1$ in $C$ such that $C$ is isometric to $\mathbb{R}^{\dim_\mathbb{R} C}_p$ and for all $i$, $\|x_i - y_i\| \leq \frac1N$.
\end{lemma}

\begin{proof}
Let $x_1$,..., $x_n$ in $X$ of norm exactly $1$, and $N \in \mathbb{N}_{\geq 1}$. We apply our hypothesis (i.e. the conclusion of Lemma~\ref{l1-2}) for these $x_i$'s and we set $N$ to be $2N$. We therefore obtain a subspace $C \subseteq X$ and $y'_1$,..., $y'_n$ in $C$ of norm at most $1$ such that $C$ is of dimension at most $(8nN + 1)^n$ (note that we no longer care about this), it is isometric to $\mathbb{R}^{\dim_\mathbb{R} C}_p$ and for all $i$, $\|x_i - y'_i\| \leq \frac1{2N}$. For each $i$, we have that $1 = \|x_i\| \leq \|y'_i\| + \|x_i - y'_i\| \leq \|y'_i\| + \frac1{2N}$, from which we get that $\|y'_i\| \geq 1 - \frac1{2N} > 0$. We may therefore set $\alpha_i := \frac1{\|y'_i\|}$ and $y_i:=\alpha_i y'_i$. Those vectors are of norm $1$ and still in $C$, so what remains to be shown is that for each $i$, $\|x_i - y_i\| \leq \frac1N$.

For each $i$, $\|y'_i\| \leq 1$, so $\alpha_i -1 \geq 0$. Then from the relation $\|y'_i\| \geq 1 - \frac1{2N}$ obtained above, we get that $\frac{\alpha_i -1}{\alpha_i} \leq \frac1{2N}$, so:
$$\|x_i - y_i\| = \|x_i - \alpha_iy'_i\| \leq \|x_i - y'_i\| + \|y'_i - \alpha_iy'_i\| \leq \frac1{2N} + (\alpha_i -1)\|y'_i\| =  \frac1{2N} + \frac{\alpha_i -1}{\alpha_i} \leq \frac1N.$$
\end{proof}

We may now state the result that we were striving for.

\begin{theorem}\label{Banach-char-final}
Let $p > 1$. A Banach space $X$ is isometric to some $L_p(\mu)$ space iff for all $x_1$,..., $x_n$ in $X$ of norm at most $1$ and for all $N \in \mathbb{N}_{\geq 1}$, there is a subspace $C \subseteq X$ and $y_1$,..., $y_n$ in $C$ of norm at most $1$ such that $C$ is of dimension at most $(4nN + 1)^n$, it is isometric to $\mathbb{R}^{\dim_\mathbb{R} C}_p$ and for all $i$, $\|x_i - y_i\| \leq \frac1N$.
\end{theorem}

\begin{proof}
The ``only if'' direction is simply Lemma~\ref{l1-2}.

To prove the ``if'' direction, we first apply Lemma~\ref{l2-new} and then use the perturbation argument laid out in \cite[p. 198]{LinTza73} to infer that the space under discussion is actually a $\mathcal{L}_{p,1+\varepsilon}$ space for all $\varepsilon > 0$, so that we may finish the equivalence proof by applying Theorem~\ref{Banach-char}. In the following, for the sake of self-containedness, we detail this perturbation argument.

By Lemma~\ref{l2-new}, we know that for all $x_1$,..., $x_n$ in $X$ of norm exactly $1$, and for all $N \in \mathbb{N}_{\geq 1}$, there is a subspace $C \subseteq X$ and $y_1$,..., $y_n$ of norm exactly $1$ in $C$ such that $C$ is isometric to $\mathbb{R}^{\dim_\mathbb{R} C}_p$ and for all $i$, $\|x_i - y_i\| \leq \frac1N$. It is now sufficient to prove that for all $\varepsilon >0$ and all finite-dimensional $Y \subseteq X$ there is a finite-dimensional $Z \subseteq X$ with $Y \subseteq Z$ and there is a linear isomorphism $T :  \mathbb{R}^{\dim_\mathbb{R} Z}_p \to Z$ with $\|T\|\|T^{-1}\| \leq 1+\varepsilon$.

Let $\varepsilon > 0$ and let $Y$ be a finite-dimensional subspace of $X$. Take $x_1$,..., $x_n$ of norm exactly 1 that form a basis for $Y$. Since all norms on a finite-dimensional space are equivalent, we have that there is a $K \geq 1$ such that for all families of scalars $(\lambda_i)$,
$$K^{-1} \max_i |\lambda_i| \leq \left\| \sum_i \lambda_i x_i \right\| \leq K \max_i |\lambda_i|.$$
Set $\delta:=\frac\varepsilon{2+\varepsilon}$, so $\delta \in (0,1)$ and $\frac{1+\delta}{1-\delta}=1+\varepsilon$. Take $N \in \mathbb{N}$ such that $\frac1N \leq \frac{\delta}{2nK}$. Then there is a subspace $C \subseteq X$ and $y_1$,..., $y_n$ of norm exactly $1$ in $C$ such that $C$ is isometric to $\mathbb{R}^{\dim_\mathbb{R} C}_p$ and for all $i$, $\|x_i - y_i\| \leq \frac{\delta}{2nK}$. Then, using the triangle inequality and that $K \geq 1$ and $\delta <1$, we get that for all families of scalars $(\lambda_i)$,
$$(2K)^{-1} \max_i |\lambda_i| \leq \left\| \sum_i \lambda_i y_i \right\| \leq 2K \max_i |\lambda_i|.$$
From this it follows that the $y_i$'s are linearly independent. Therefore, one can use the Hahn-Banach theorem on the subspace of $C$ generated by the $y_i$'s to obtain $n$ continuous linear functionals on $X$, $x_1^*$,..., $x_n^*$, such that for each $i$, $\|x_i^*\| \leq 2K$ and for each $i$ and $j$, $x_i^*(y_j)=\delta_{ij}$. Define now the operator $U: C \to X$, for all $h \in C$, by:
$$U(h):= h + \sum_i x_i^*(h)(x_i-y_i).$$
Set $Z$ to be the image of $U$, so $Z$ is finite-dimensional. Since, clearly, for all $i$, $U(y_i)=x_i$, $Y \subseteq Z$. Set now $m:=\dim_\mathbb{R} C$ and let $V: \mathbb{R}^{m}_p \to C$ be an isometry. Then $\{UV(e_1),...,UV(e_m)\}$ span $Z$, so there are $i_1$,..., $i_l$ such that $\{UV(e_{i_1}),..., UV(e_{i_l})\}$ is a basis for $Z$. If we identify the vector space spanned by $\{e_{i_1},...,e_{i_l}\}$ with $\mathbb{R}^l_p$, we may take $T$ to be the restriction of $UV$ to this space. Then $T$ is a linear isomorphism from $\mathbb{R}^{\dim_\mathbb{R} Z}_p $ to $Z$ and a simple calculation that uses the definition of $U$ shows that $\|T\| \leq 1 + \delta$ and $\|T^{-1}\| \leq \frac1{1-\delta}$, so  $\|T\|\|T^{-1}\| \leq \frac{1+\delta}{1-\delta} =1+\varepsilon$.
\end{proof}

The advantage of the condition obtained in the above theorem is that it is both intrinsic and quantitative, therefore amenable to a logical axiomatization.

\begin{table}[ht!]
\begin{center}
\begin{tabular}{ | c |}
\hline
\ \\
$\psi_m(\underline{z}) := \forall \underline{a} \left( \left\| \sum_{i=1}^m a_i z_i \right\| = \left( \sum_{i=1}^m |a_i|^p \right) ^{\frac1p} \right)$\\[2mm]
$\psi'_{m,n}(\underline{y},\underline{z}) := \bigwedge_{k=1}^n \left( \exists \underline{\lambda} \left( y_k = \sum_{i=1}^m \lambda_i z_i \right) \right)$\\[2mm]
$\psi''_{n,N}(\underline{x},\underline{y}) := \bigwedge_{k=1}^n \left( \|x_k - y_k\| \leq \frac1{N+1} \land \|y_k\| \leq 1 \right) $\\[2mm]
$\varphi_{n,m,N}(\underline{x}) := \exists \underline{y} \exists \underline {z} \left( \psi_m(\underline{z}) \land \psi'_{m,n}(\underline{y},\underline{z}) \land \psi''_{n,N}(\underline{x},\underline{y})  \right) $\\[2mm]
$\phi_{n,N}(\underline{x}) := \bigvee_{0 \leq m \leq (4nN + 1)^n} \varphi_{n,m,N}(\underline{x})$\\[2mm]
$ A_{n,N} := \forall \underline{x} \left( \left(\bigwedge_{k=1}^n \|x_k\|\leq 1 \right) \to \phi_{n,N} (\underline{x}) \right)$\\[2mm]
\ \\
\hline
\end{tabular}
\caption{A first axiomatization.}\label{tabel-1}
\end{center}
\end{table}

Table~\ref{tabel-1} shows one such axiomatization (into a crude first-order-like language), i.e. the characterization of the space is expressed by the simultaneous validity of all $A_{n,N}$ sentences. With that in mind, by closely examining the formulas, one can easily see that they represent a straightforward translation of the condition from before.

\begin{table}[ht!]
\begin{center}
\begin{tabular}{ | c |}
\hline
\ \\
$\psi(m,z) := \forall a^{1(0)} \left( \left\| \sum_{i=1}^m |a(i)|_\mathbb{R} \cdot_X z(i) \right\| =_\mathbb{R} \left( \sum_{i=1}^m |a(i)|^p_\mathbb{R} \right)^{1/p} \right)$\\[2mm]
$\psi'(m,n,y,z,\lambda) := \forall k \preceq_0 (n-1) \left( y(k+1) =_X \sum_{i=1}^m \lambda(k+1,i) \cdot_X z(i) \right)$\\[2mm]
$\psi''(n,N,x,y) := \forall k \preceq_0 (n-1) \left( \left\| \widetilde{x(k+1)} - y(k+1) \right\| \leq_\mathbb{R} \frac1N \land \|y(k+1)\| \leq_\mathbb{R} 1 \right)$\\[2mm]
$\varphi(n,m,N,x,y,z,\lambda) := \psi(m,z) \land \psi'(m,n,y,z,\lambda) \land \psi''(n,N,x,y)$\\[2mm]
$B := \forall n^0, N^0 \geq 1 \forall x^{X(0)} \exists y, z \preceq_{X(0)} 1_{X(0)} \exists \lambda^{1(0)(0)} \in [-1,1] \exists m \preceq_0 (4nN+1)^n$\\[2mm]
$\varphi(n,m,N,x,y,z,\lambda)$\\[2mm]
\ \\
\hline
\end{tabular}
\caption{The $\Delta$-axiomatization.}\label{tabel-2}
\end{center}
\end{table}

Table~\ref{tabel-2}, where we have used some of the notations from \cite[Definitions 7.9 and 7.10]{GunKoh16}, shows how one may translate the infinite family of axioms $A_{n,N}$ into the one axiom $B$ which is, like the one in \cite{GunKoh16}, representable as a $\Delta$-sentence. Let us see some details of the translation. Firstly, we remark that the operation $\widetilde{v} := \frac{v}{\max\{\|v\|,1\}}$ that we used excused us from writing the antecedent of $A_{n,N}$. Then we see that by substituting into $\psi_m(\underline{z})$ all $\lambda_i$'s with $0$, except for one which we set to $1$, we obtain the fact that all $z_i$'s are of norm one. We have also postulated that all $y_k$'s are of norm at most $1$. Thus, if we have, as in $\psi'_{m,n}(\underline{y},\underline{z})$, that for a given $k$:
$$y_k = \sum_{i=1}^m \lambda_i z_i,$$
the formula $\psi_m(\underline{z})$ tells us further that:
$$1 \geq \|y_k\| \geq \left\|\sum_{i=1}^m \lambda_i z_i\right\| = \left( \sum_{i=1}^m |\lambda_i|^p \right) ^{\frac1p},$$
from which we get that each such $\lambda_i$ is in the interval $[-1,1]$. These results allow us to correspondingly bound the $y$, the $z$ and the $\lambda$ (which are now properly functionals) in the axiom $B$. In the process, we have used (cf. \cite[p. 93]{Koh08} and \cite[Definition 7.9.2]{GunKoh16}) the notation
$$\exists \lambda^{1(0)(0)} \in [-1,1] $$
for
$$\exists \lambda \preceq_{1(0)(0)} \lambda k,i.(\lambda n.j(2^{n+3}+1,2^{n+2}-1)).$$
Another such bounding comes from the $(4nN+1)^n$ established before (i.e. here it matters that the characterization is quantitative), which helped us eliminate the potentially infinite disjunction in Table~\ref{tabel-1} (where such constraints were not yet relevant) and the unbounded existential quantifier in Table~\ref{tabel-2} (which would have hindered us in presenting the axiom $B$ as a $\Delta$-sentence). As a curiosity, we note that choosing to present $B$ as a single axiom and not as an infinite schema like in Table~\ref{tabel-1}, i.e. taking advantage of the arithmetic already present in the framework, adds a bit of strength to the system, given the fact that we do not work here with any sort of $\omega$-rule.

We denote by $\mathcal{A}^\omega[X,\|\cdot\|,\mathcal{C},L^p]$ the extension of the system $\mathcal{A}^\omega[X,\|\cdot\|,\mathcal{C}]$ by the constant $c_p$ of type $1$, together with the axiom $1_\mathbb{R} \leq_\mathbb{R} c_p$ and the axiom $B$ from above. From the above discussion, the following soundness theorem holds.

\begin{theorem}[{cf. \cite[Propositions 3.5 and 7.12]{GunKoh16}}]\label{sound}
Let $X$ be a Banach space and $p \geq 1$. Denote by $\mathcal{S}^{\omega,X}$ its associated set-theoretic model and let the constant $c_p$ in our extended signature take as a value the canonical representation of the real number $p$. Then $\mathcal{S}^{\omega,X}$ is a model of $\mathcal{A}^\omega[X,\|\cdot\|,\mathcal{C},L^p]$ iff $X$ is isomorphic to some $L^p(\Omega, \mathcal{F}, \mu)$ space.
\end{theorem}

In a parallel way to the one suggested in \cite{HenRay07}, by some similar arguments to the ones used above to construct the required higher-order system, one could perform reasonable transformations to the formulas in Table~\ref{tabel-1}, obtaining a new, concrete proof of the following classical result of Krivine and Henson.

\begin{theorem}\label{pb}
The subclass of Banach spaces which are isomorphic to spaces of the form $L^p(\mu)$ is axiomatizable in positive-bounded logic.
\end{theorem}

Analogously to the treatment done in \cite{GunKoh16} for the classes of Banach lattices, we may now state the corresponding metatheorem for the system devised above.

\begin{theorem}[{Logical metatheorem for $L^p(\mu)$ Banach spaces, cf. \cite[Theorems 5.13 and 7.13]{GunKoh16}}]\label{S3thm}
Let $\rho \in \text{{\bf T}}^X$ be an admissible type. Let $B_\forall(x,u)$ be a $\forall$-formula with at most $x,u$ free and $C_\exists(x,v)$ an $\exists$-formula with at most $x,v$ free. Let $\Delta$ be a set of $\Delta$-sentences. Suppose that:
$$\mathcal{A}^\omega[X,\|\cdot\|,\mathcal{C},L^p] + \Delta \vdash \forall x^\rho (\forall u^0 B_\forall(x,u) \to \exists v^0 C_\exists(x,v)).$$
Then one can extract a partial functional $\Phi : S_{\widehat{\rho}} \rightharpoonup \mathbb{N}$, whose restriction to the strongly majorizable functionals of $S_{\widehat{\rho}}$ is a bar-recursively computable functional of $\mathcal{M}^\omega$, such that for all $L^p(\mu)$ Banach spaces $(X,\|\|)$ having the property that any associated set-theoretic model of it satisfies $\Delta$, we have that for all $x \in S_\rho$ and $x^* \in S_{\widehat{\rho}}$ such that $x^* \gtrsim_\rho x$, the following holds:
$$\forall u \leq \Phi(x^*) B_\forall(x,u) \to \exists v \leq \Phi(x^*) C_\exists(x,v).$$
All the additional considerations from Theorem~\ref{S2thm} also apply here.
\end{theorem}

\begin{proof}
This theorem extends Theorem~\ref{S2thm}. The two additional axioms are $\Delta$-axioms, and the constant $c_p$ is majorized (as in \cite[Lemma 17.8]{Koh08}) by $M(b) :=  \lambda n.j(b2^{n+2},2^{n+1} -1)$, where $j$ is the Cantor pairing function and $b \in \mathbb{N}$ such that $b \geq p$ (e.g., $b:= \lceil (c_p(0))_\mathbb{Q} \rceil +1$). We note that the $\Phi$ depends on $p$ only via this upper bound $b$.
\end{proof}

\section{The derivation of the modulus of uniform convexity}\label{sec:uc}

The axiomatization that we have just obtained has, essentially, the form of a comparison principle with respect to the $p$-normed Euclidean spaces. This suggests that it may be particularly application-friendly. Let us see why this is the case. Suppose that we have an existing mathematical theorem regarding $L^p$ spaces. The particularization of the proof to the Euclidean case is likely to be easily derivable in our higher systems of arithmetic (with the possible addition of universal lemmas), since statements about integrals are reduced to statements about sums and powers of real numbers. The second step would be to translate the result along the $\varepsilon$-close approximation of our characterization, a translation involving a sequence of boundings which is likely to leave the original statement intact if it is well-behaved enough. We shall now illustrate this general strategy on a classical result on $L^p$ spaces.

Uniform convexity is a fundamental notion in the theory of Banach spaces, introduced by Clarkson \cite{Cla36} in 1936. As per \cite[Section 6.4]{GunKoh16}, the property can be formalized as:
$$\forall k^0 \exists n^0 \forall x_1, x_2 \preceq_X 1_X \left(\left\|\frac12(x_1 + x_2)\right\| \geq 1-2^{-n} \to \|x_1-x_2\|<2^{-k}\right).$$
and it is suitable for bound extraction. We note that, in the above statement, like in the definition of the convergence of a sequence, a bound (for $n$, in this case) is also a witness. Also, with the logical issues now resolved, we note that, for the ease of understanding, we shall work with $\varepsilon$-style characterizations. Therefore, following \cite[Section 2.1]{KohLeu12}, we define a {\bf modulus of uniform convexity} for a Banach space to be a function $\eta : (0,2] \to (0,\infty)$ such that for any $\varepsilon > 0$ and any $x_1$ and $x_2$ with $\|x_1\| \leq 1$, $\|x_2\| \leq 1$ and $\|x_1-x_2\|\geq \varepsilon$, we have that
$$\left\|\frac12(x_1 + x_2)\right\| \leq 1-\eta(\varepsilon).$$
We make the observation that what is usually called ``the'' modulus of uniform convexity of a space is the ``optimal'' such modulus, i.e. for each $\varepsilon > 0$ we take as $\eta(\varepsilon)$ the greatest value of $\delta$ that works for all suitable $x_1$, $x_2$, i.e. the minimum of the expression $1-\left\|\frac12(x_1 + x_2)\right\|$. The goal of this section is to derive a modulus of uniform convexity for $L^p(\mu)$ spaces using only the axiomatization established in the previous section. We will consider, for simplicity, $p \geq 2$, i.e. we add the additional admissible axiom $2 \leq_\mathbb{R} c_p$ to our system. For this case, the modulus of uniform convexity was already computed by Clarkson (see \cite[p. 403]{Cla36}) and later shown to be optimal by Hanner \cite[Theorem 2]{Han56}.

We begin with some results of real analysis. The following lemma and corollary are standard in the literature.

\begin{lemma}
For all $x_1, x_2 \geq 0$, $x_1^p + x_2^p \leq (x_1^2 + x_2^2)^{p/2}$.
\end{lemma}

\begin{proof}
The case $x_2 = 0$ is clear. If $x_2 \neq 0$, we can divide by $x_2^p$ and we notice that we only have to prove that for all $t \geq 0$, $t^p + 1 \leq (t^2 + 1)^{p/2}$. Consider the function $f: \mathbb{R} \to \mathbb{R}$, defined, for all $t$, by $f(t):= (t^2 + 1)^{p/2} - t^p - 1$. Since $f'(t) = \frac{p}2(t^2 + 1)^{(p/2) - 1} \cdot 2t - pt^{-1} \geq pt^{p-2} \cdot t - pt^{p-1} = 0$ and $f(0) = 0$, we obtain that for all $t$, $f(t) \geq 0$, and hence the conclusion.
\end{proof}

\begin{corollary}\label{coro}
For all $a, b \in \mathbb{R}$, $\left|\frac{a+b}2\right|^p + \left|\frac{a-b}2\right|^p \leq \frac12(|a|^p + |b|^p)$.
\end{corollary}

\begin{proof}
We substitute into the above lemma $x_1 := \left|\frac{a+b}2\right|$ and $x_2 := \left|\frac{a-b}2\right|$. Since $\left|\frac{a+b}2\right|^2 + \left|\frac{a-b}2\right|^2 = \frac12(a^2 + b^2)$, we obtain that:
\begin{align*}
\left|\frac{a+b}2\right|^p + \left|\frac{a-b}2\right|^p &\leq \left( \frac12(a^2 + b^2) \right) ^{p/2} \\
&\leq \frac12 ((a^2)^{p/2} + (b^2)^{p/2}) \\
&= \frac12(|a|^p + |b|^p),
\end{align*}
where the last inequality follows from the convexity of the function $t \mapsto t^p$ on $(0, \infty)$, for any $p \geq 2$.
\end{proof}

Set, now, for all $a, d \in (0,1)$, $\sigma(a,d):=a-(1-((1-a^p)^{1/p} + d)^p )^{1/p}$.

\begin{lemma}
For all $a,d \in (0,1)$ with $d < 1-(1-a^p)^{1/p}$, $\sigma(a,d)$ is defined and strictly positive.
\end{lemma}

\begin{proof}
For the first part, we see that the condition implies that
$$(1-a^p)^{1/p} + d < 1,$$
so
$$1-((1-a^p)^{1/p} + d)^p \geq 0,$$
which is what we need for the last $1/p$'th power in the definition of $\sigma(a,d)$ to make sense.

For the second part, since $d>0$, we have that $(1-a^p)^{1/p} < (1-a^p)^{1/p} + d$, so
$$1-a^p < ((1-a^p)^{1/p} + d)^p.$$
From that we successively obtain:
$$a^p > 1- ((1-a^p)^{1/p} + d)^p,$$
$$a > (1- ((1-a^p)^{1/p} + d)^p)^{1/p},$$
$$a - (1- ((1-a^p)^{1/p} + d)^p)^{1/p} > 0.$$
\end{proof}

\begin{lemma}\label{lsigma}
For all $a,d \in (0,1)$ with $d < 1-(1-a^p)^{1/p}$ and all $\delta \in (0, \sigma(a,d))$, we have that:
$$(1 - (a-\delta)^p)^{1/p} \leq (1-a^p)^{1/p} + d.$$
\end{lemma}

\begin{proof}
Clearly $\sigma(a,d) < a$, so $(a-\delta)^p$ is well-defined. Now, since
$$\delta \leq a-(1-((1-a^p)^{1/p} + d)^p )^{1/p},$$
we obtain, successively, that:
$$ a - \delta \geq (1-((1-a^p)^{1/p} + d)^p )^{1/p},$$
$$ (a- \delta)^p \geq 1-((1-a^p)^{1/p} + d)^p,$$
$$ 1 - (a- \delta)^p \leq ((1-a^p)^{1/p} + d)^p,$$
$$ (1 - (a- \delta)^p)^{1/p} \leq (1-a^p)^{1/p} + d.$$
\end{proof}

Note that the statements of Corollary~\ref{coro} and Lemma~\ref{lsigma} are universal and therefore it is admissible to add them as supplementary axioms -- denote them by $C_1$ and $C_2$. These statements, although provable in our higher-typed system, do not contribute to any information that might be extracted out of a proof -- this is why we may use them freely, with no concern for their origin. In addition, they refer only to real numbers, not to any abstract types, therefore their presence does not bias our effort to prove the suitability of our axiomatization. The conclusion of our enterprise, the modulus of uniform convexity, would also transform the convexity statement into such a universal sentence, which {\it concerns} abstract types and which could therefore only {\it afterwards} be added to the system in order to analyze a subsequent proof. That being said, we are now in a position to state the main theorem of this section.

\begin{theorem}\label{uconv}
Provably in the system $\mathcal{A}^\omega[X,\|\cdot\|,\mathcal{C},L^p] + \{2 \leq_{\mathbb{R}} c_p; C_1; C_2\}$, the function $\eta : (0,2] \to (0,\infty)$, defined, for any $\varepsilon >0$, by $\eta(\varepsilon) := 1- (1-(\frac{\varepsilon}2)^p)^{1/p}$, is a modulus of uniform convexity.
\end{theorem}

\begin{proof}
Let $\varepsilon > 0$. Take $x_1, x_2 \in X$ with $\|x_1\|, \|x_2\| \leq 1$ and $\|x_1-x_2\| \geq \varepsilon$. Let $c \in (0,1)$ such that $c < 1-\left(1-\left(\frac\varepsilon2\right)^p\right)^{1/p}$, so that $\sigma\left(\frac{\varepsilon}2,\frac{c}2\right)$ is well-defined. Set $\delta := \min\left\{ \frac{c}2, \frac{\sigma(\frac{\varepsilon}2,\frac{c}2)}2 \right\}$. Take $y_1, y_2$, $z_1,\ldots,z_m$ like in our axiomatization (e.g., from Table~\ref{tabel-1}) such that for all $k \in \{1,2\}$,
$$\|x_k - y_k\| \leq \delta,\quad \|y_k\|\leq 1.$$
Write now:
$$y_1 = \sum_{i=1}^m \lambda_i z_i,\quad y_2 = \sum_{i=1}^m \mu_i z_i.$$
We have that:
\begin{align*}
\left\|\frac{y_1+y_2}2\right\|^p + \left\|\frac{y_1-y_2}2\right\|^p &= \left\|\sum_{i=1}^m \frac{\lambda_i + \mu_i}2 z_i\right\|^p + \left\|\sum_{i=1}^m \frac{\lambda_i - \mu_i}2 z_i\right\|^p \\
&= \sum_{i=1}^m \left( \left|\frac{\lambda_i+\mu_i}2\right|^p + \left|\frac{\lambda_i-\mu_i}2\right|^p \right) \\
&\leq \frac12 \sum_{i=1}^m (|\lambda_i|^p + |\mu_i|^p) \\
&= \frac12 (\|y_1\|^p + \|y_2\|^p ) \\
&\leq 1.
\end{align*}
Assume that $\|y_1-y_2\| \geq \rho$. Then we get that
$$\left\|\frac{y_1+y_2}2\right\| \leq \left(1-\left(\frac{\rho}2\right)^p\right)^{1/p}.$$
Incidentally, what we have shown above is the validity of $\eta$ as a modulus of uniform convexity for the $\mathbb{R}^m_p$ spaces (with $p \geq 2$).

Note that:
$$\varepsilon \leq \|x_1 - x_2\| \leq \|x_1 - y_1\| + \|y_1 - y_2\| + \|y_2 - x_2\| \leq \|y_1 - y_2\| + 2\delta$$
and hence we may take $\rho := \varepsilon - 2\delta > 0$ (since $\delta < \sigma(\frac{\varepsilon}2,\frac{c}2) < \frac{\varepsilon}2$). We have obtained that:
$$\left\|\frac{y_1+y_2}2 \right\| \leq \left(1-\left(\frac{\varepsilon}2 - \delta\right)^p\right)^{1/p}.$$

On the other hand,
$$\|x_1 + x_2\| \leq \|y_1+y_2\| + \|(x_1+x_2)-(y_1+y_2)\| \leq \|y_1+y_2\| + \|x_1 - y_1\| + \|x_2 - y_2\| \leq \|y_1 + y_2\| + 2\delta,$$
so
$$\left\|\frac{x_1 + x_2}2\right\| \leq \left\|\frac{y_1+y_2}2 \right\| + \delta \leq \left(1-\left(\frac{\varepsilon}2 - \delta\right)^p \right)^{1/p} + \delta.$$

Since $0<\delta < \sigma(\frac\varepsilon2,\frac c 2)$, we have that:
$$\left(1-\left(\frac{\varepsilon}2 - \delta\right)^p \right)^{1/p} \leq \left(1-\left(\frac{\varepsilon}2\right)^p \right)^{1/p} + \frac c 2.$$
Also, we know that $\delta \leq \frac c 2$, so we finally obtain that:
$$\left\|\frac{x_1 + x_2}2\right\| \leq \left(1-\left(\frac{\varepsilon}2\right)^p \right)^{1/p} + c.$$

Now, since $c$ could be chosen arbitrarily close to 0, we may prove in our system that
$$\left\|\frac{x_1 + x_2}2\right\| \leq \left(1-\left(\frac{\varepsilon}2\right)^p \right)^{1/p},$$
showing, indeed, that $\eta$ is a modulus of uniform convexity.
\end{proof}

\section{Acknowledgments}

The author is grateful to Ulrich Kohlenbach and to Lauren\c tiu Leu\c stean for the helpful discussions and suggestions regarding the subject matter and the final form of the paper, and to the anonymous reviewers who pointed out connections to results established in the literature.

This work was supported by a grant of the Romanian National Authority for Scientific Research, CNCS - UEFISCDI, project number PN-II-ID-PCE-2011-3-0383.

\end{document}